\documentclass{amsart}
\usepackage{amsmath}
\usepackage{amsfonts}

\newtheorem{theorem}{Theorem}
\theoremstyle{plain}

\newtheorem{corollary}{Corollary}

\newtheorem{definition}{Definition}

\newtheorem{lemma}{Lemma}

\newtheorem{remark}{Remark}

\numberwithin{equation}{section}
\input{tcilatex}

\begin{document}
\title[On Hermite-Hadamard type integral inequalities]{On Hermite-Hadamard
type integral inequalities for strongly $\varphi _{h}$-convex functions}
\author{Mehmet Zeki SARIKAYA}
\address{Department of Mathematics, \ Faculty of Science and Arts, D\"{u}zce
University, D\"{u}zce-TURKEY}
\email{sarikayamz@gmail.com}
\author{Kubilay OZCELIK}
\address{Department of Mathematics, \ Faculty of Science and Arts, D\"{u}zce
University, D\"{u}zce-TURKEY}
\email{kubilayozcelik@windowslive.com }

\begin{abstract}
In this paper,  using functions whose derivatives absolute values are
strongly $\varphi _{h}$-convex with modulus $c>0$, we obtained new
inequalities releted to the right and left side of Hermite-Hadamard
inequality by using new integral identities
\end{abstract}

\maketitle

\section{Introduction}

The inequalities discovered by C. Hermite and J. Hadamard for convex
functions are very important in the literature (see, e.g.,\cite{dragomir1},%
\cite[p.137]{pecaric}). These inequalities state that if $f:I\rightarrow 
\mathbb{R}$ is a convex function on the interval $I$ of real numbers and $%
a,b\in I$ with $a<b$, then 
\begin{equation}
f\left( \frac{a+b}{2}\right) \leq \frac{1}{b-a}\int_{a}^{b}f(x)dx\leq \frac{%
f\left( a\right) +f\left( b\right) }{2}.  \label{E1}
\end{equation}%
The inequality (\ref{E1}) has evoked the interest of many mathematicians.
Especially in the last three decades numerous generalizations, variants and
extensions of this inequality have been obtained, to mention a few, see (%
\cite{bakula}-\cite{set2}) and the references cited therein.

Let $I$ be an interval in $\mathbb{R}$ and $h:(0,1)\rightarrow (0,\infty )$
be a given function. A function $f:I\rightarrow \lbrack 0,\infty )$ is said
to be $h$-convex if%
\begin{equation}
f(tx+(1-t)y)\leq h(t)f(x)+h(1-t)f(y)  \label{h1}
\end{equation}%
for all $x,y\in I$ and $t\in \left( 0,1\right) $ \cite{varosanec}$.$ This
notion unifies and generalizes the known classes of functions, $s$-convex
functions, Gudunova-Levin functions and $P$-functions, which are obtained by
putting in (\ref{h1}), $h(t)=t,\ h(t)=t^{s},\ h(t)=\frac{1}{t},$ and $%
h(t)=1, $ respectively. Many properties of them can be found, for instance,
in \cite{dragomir3},\cite{dragomir4},\cite{angu},\cite{sarikaya},\cite%
{sarikaya1},\cite{sarikaya3},\cite{varosanec}.

Let us consider a function $\varphi :[a,b]\rightarrow \lbrack a,b]$ where $%
[a,b]\subset \mathbb{R}$. Youness have defined the $\varphi $-convex
functions in \cite{youness}:

\begin{definition}
A function $f:[a,b]\rightarrow \mathbb{R}$ is said to be $\varphi $- convex
on $[a,b]$ if for every two points $x\in \lbrack a,b],y\in \lbrack a,b]$ and 
$t\in \lbrack 0,1]$ the following inequality holds:%
\begin{equation*}
f(t\varphi (x)+(1-t)\varphi (y))\leq tf(\varphi (x))+(1-t)f(\varphi (y)).
\end{equation*}
\end{definition}

Obviously, if function $\varphi $ is the identity, then the classical
convexity is obtained from the previous definition. Many properties of the $%
\varphi $-convex functions can be found, for instance, in \cite{cristescu}, 
\cite{cristescu1},\cite{youness},\cite{sarikaya2},\cite{sarikaya3}.

Moreover in \cite{cristescu1}, Cristescu have presented a version
Hermite-Hadamard type inequality for the $\varphi $-convex functions as
follows:

\begin{theorem}
\label{zz} If a function $f:[a,b]\rightarrow \mathbb{R}$ is $\varphi $-
convex for the continuous function $\varphi :\left[ a,b\right] \rightarrow %
\left[ a,b\right] $, then%
\begin{equation}
f\left( \frac{\varphi (a)+\varphi (b)}{2}\right) \leq \frac{1}{\varphi
(b)-\varphi (a)}\dint\limits_{\varphi (a)}^{\varphi (b)}f(x)dx\leq \frac{%
f(\varphi (a))+f(\varphi (b))}{2}.  \label{E2}
\end{equation}
\end{theorem}

Recall also that a function $f:I\rightarrow \mathbb{R}$ is called strongly
convex with modulus $c>0,$ if%
\begin{equation*}
f\left( tx+\left( 1-t\right) y\right) \leq tf\left( x\right) +\left(
1-t\right) f\left( y\right) -ct(1-t)(x-y)^{2}
\end{equation*}%
for all $x,y\in I$ and $t\in (0,1).$ Strongly convex functions have been
introduced by Polyak in \cite{polyak}\ and they play an important role in
optimization theory and mathematical economics. Various properties and
applicatins of them can be found in the literature see (\cite{polyak}-\cite%
{angu}) and the references cited therein.

In \cite{sarikaya4} , Sarikaya have introduced the following notion of the
strongly $\varphi _{h}$-convex functions with modulus $c>0,$ and give some
properties of them:

A function $f:D\rightarrow \lbrack 0,\infty )$ is said to be strongly $%
\varphi _{h}$-convex with modulus $c>0,$ if

\begin{eqnarray}
&&f(t\varphi (x)+(1-t)\varphi (y))  \notag \\
&&  \label{d1} \\
&\leq &h(t)f(\varphi (x))+h(1-t)f(\varphi (y))-ct(1-t)\left( \varphi
(x)-\varphi (y)\right) ^{2}  \notag
\end{eqnarray}%
for all $x,y\in D$ and $t\in \left( 0,1\right) $. In particular, if $f$
satisfies (\ref{d1}) with $h(t)=t$, $h(t)=t^{s}\ (s\in \left( 0,1\right) ),\
h(t)=\frac{1}{t},$ and $h(t)=1,$ then $f$ is said to be strongly $\varphi $%
-convex, strongly $\varphi _{s}$-convex, strongly $\varphi $-Gudunova-Levin
function and strongly $\varphi $-$P$-function, respectively. The notion of $%
\varphi _{h}$-convex function corresponds to the case $c\rightarrow 0$.

In this article, using functions whose derivatives absolute values are
strongly $\varphi _{h}$-convex with modulus $c>0$, we obtained new
inequalities releted to the right and left side of Hermite-Hadamard
inequality by using new integral identities. In particular if $\varphi =0$
is taken as, our results obtained reduce to the Hermite-Hadamard type
inequality for classical convex functions.

\section{Main Results}

In order to prove our main results, we establish a important integral
identity as follows:

\begin{lemma}
\label{z1} Let $f:I\subset 
%TCIMACRO{\U{211d} }%
%BeginExpansion
\mathbb{R}
%EndExpansion
\rightarrow 
%TCIMACRO{\U{211d} }%
%BeginExpansion
\mathbb{R}
%EndExpansion
$ be a differentiable mapping on $I^{0}$ where $a,b\in I$ with $a<b$ and $%
\varphi :\left[ a,b\right] \rightarrow \left[ a,b\right] .$ If $f^{\prime }\ 
$is a Lebesgue integrable function, then the following equality holds;%
\begin{eqnarray}
&&\frac{f\left( \varphi (a)\right) +f\left( \varphi (b)\right) }{2}-\frac{1}{%
(\varphi (b)-\varphi (a))}\int\limits_{\varphi (a)}^{\varphi (b)}f\left(
x\right) dx  \label{s1} \\
&=&\dfrac{(\varphi (b)-\varphi (a))}{2}\int\limits_{0}^{1}(2t-1)\left[
f^{\prime }\left( t\varphi (b)+\left( 1-t\right) \varphi (a)\right)
+ct(1-t)(\varphi (b)-\varphi (a))^{2}\right] dt.  \notag
\end{eqnarray}
\end{lemma}

\begin{proof}
By integration by parts, we can state:%
\begin{eqnarray*}
I &=&\int\limits_{0}^{1}(2t-1)\left[ f^{\prime }\left( t\varphi (b)+\left(
1-t\right) \varphi (a)\right) +ct(1-t)(\varphi (b)-\varphi (a))^{2}\right] dt
\\
&& \\
&=&(2t-1)\frac{f\left( t\varphi (b)+\left( 1-t\right) \varphi (a)\right) }{%
(\varphi (b)-\varphi (a))}\underset{0}{\overset{1}{\mid }}-\frac{2}{(\varphi
(b)-\varphi (a))}\int\limits_{0}^{1}f\left( t\varphi (b)+\left( 1-t\right)
\varphi (a)\right) dt \\
&& \\
&=&\frac{f\left( \varphi (b)\right) +f\left( \varphi (a)\right) }{(\varphi
(b)-\varphi (a))}-\frac{2}{(\varphi (b)-\varphi (a))}\int\limits_{0}^{1}f%
\left( t\varphi (b)+\left( 1-t\right) \varphi (a)\right) dt.
\end{eqnarray*}%
Using the change of the variable $x=t\varphi (b)+\left( 1-t\right) \varphi
(a)$ for $t\in \left[ 0,1\right] ,$ which gives%
\begin{equation}
I=\frac{f\left( \varphi (b)\right) +f\left( \varphi (a)\right) }{(\varphi
(b)-\varphi (a))}-\frac{2}{(\varphi (b)-\varphi (a))^{2}}\int\limits_{%
\varphi (a)}^{\varphi (b)}f\left( x\right) dx.  \label{2}
\end{equation}%
Multiplying the both sides of (\ref{2}) by $\dfrac{(\varphi (b)-\varphi (a))%
}{2}$, we obtain%
\begin{equation*}
\dfrac{(\varphi (b)-\varphi (a))}{2}I=\frac{f\left( \varphi (b)\right)
+f\left( \varphi (a)\right) }{2}-\frac{1}{(\varphi (b)-\varphi (a))}%
\int\limits_{\varphi (a)}^{\varphi (b)}f\left( x\right) dx
\end{equation*}%
which is required.
\end{proof}

\begin{theorem}
\label{t1} Let $h:\left( 0,1\right) \rightarrow \left( 0,\infty \right) $ be
a given function. Let $f:I\subset 
%TCIMACRO{\U{211d} }%
%BeginExpansion
\mathbb{R}
%EndExpansion
\rightarrow 
%TCIMACRO{\U{211d} }%
%BeginExpansion
\mathbb{R}
%EndExpansion
$ be a differentiable mapping on $I^{0}$ where $a,b\in I$ with $a<b$ and
Lebesgue integrable function$.$ If $\left\vert f^{\prime }\right\vert $ is
strongly $\varphi _{h}-$convex with respect to $c>0$ for the continuous
function $\varphi :\left[ a,b\right] \rightarrow \left[ a,b\right] $ and $%
\varphi (a)<\varphi (b),$ then the following inequality holds;%
\begin{eqnarray}
&&\left\vert \frac{f\left( \varphi (a)\right) +f\left( \varphi (b)\right) }{2%
}-\frac{1}{(\varphi (b)-\varphi (a))}\int\limits_{\varphi (a)}^{\varphi
(b)}f\left( x\right) dx\right\vert   \label{s2} \\
&&  \notag \\
&\leq &\dfrac{(\varphi (b)-\varphi (a))}{2}\left[ \left\vert f^{\prime
}(\varphi (b))\right\vert +\left\vert f^{\prime }(\varphi (a))\right\vert %
\right] \int\limits_{0}^{1}\left\vert 2t-1\right\vert h(t)dt  \notag
\end{eqnarray}%
for all $t\in \left( 0,1\right) .$
\end{theorem}

\begin{proof}
From Lemma \ref{z1} and by using strongly $\varphi _{h}-$convexity functions
with modulus $c>0$ of $\left\vert f^{\prime }\right\vert ,$ we have%
\begin{eqnarray*}
&&\left\vert \frac{f\left( \varphi (a)\right) +f\left( \varphi (b)\right) }{2%
}-\frac{1}{(\varphi (b)-\varphi (a))}\int\limits_{\varphi (a)}^{\varphi
(b)}f\left( x\right) dx\right\vert  \\
&& \\
&\leq &\dfrac{(\varphi (b)-\varphi (a))}{2}\int\limits_{0}^{1}\left\vert
2t-1\right\vert \left[ \left\vert f^{\prime }\left( t\varphi (b)+\left(
1-t\right) \varphi (a)\right) \right\vert +ct(1-t)(\varphi (b)-\varphi
(a))^{2}\right] dt \\
&& \\
&\leq &\dfrac{(\varphi (b)-\varphi (a))}{2}\int\limits_{0}^{1}\left\vert
2t-1\right\vert \left[ h(t)\left\vert f^{\prime }(\varphi (b))\right\vert
+h(1-t)\left\vert f^{\prime }(\varphi (a))\right\vert \right] dt \\
&& \\
&=&\dfrac{(\varphi (b)-\varphi (a))}{2}\left[ \left\vert f^{\prime }(\varphi
(b))\right\vert +\left\vert f^{\prime }(\varphi (a))\right\vert \right]
\int\limits_{0}^{1}\left\vert 2t-1\right\vert h(t)dt
\end{eqnarray*}%
where using the fact that%
\begin{equation*}
\int\limits_{0}^{1}h(t)dt=\int\limits_{0}^{1}h(1-t)dt
\end{equation*}%
which completes the proof.
\end{proof}

The following inequalities are associated the right side of Hermite-Hadamard
type inequalities for strongly $\varphi $-convex, strongly $\varphi _{s}$%
-convex strongly $\varphi -P$-convex with respect to $c>0,$ respectively.

\begin{corollary}
Under the assumptions of Theorem \ref{t1} with $h(t)=t,\ t\in \left(
0,1\right) ,$ we have%
\begin{equation}
\left\vert \frac{f\left( \varphi (a)\right) +f\left( \varphi (b)\right) }{2}-%
\frac{1}{(\varphi (b)-\varphi (a))}\int\limits_{\varphi (a)}^{\varphi
(b)}f\left( x\right) dx\right\vert \leq (\varphi (b)-\varphi (a))\left( 
\frac{\left\vert f^{\prime }(\varphi (b))\right\vert +\left\vert f^{\prime
}(\varphi (a))\right\vert }{8}\right) .  \label{c1}
\end{equation}
\end{corollary}

\begin{corollary}
Under the assumptions of Theorem \ref{t1} with $h(t)=t^{s}$ $(s\in \left(
0,1\right) ),\ t\in \left( 0,1\right) ,$ we have%
\begin{eqnarray}
&&\left\vert \frac{f\left( \varphi (a)\right) +f\left( \varphi (b)\right) }{2%
}-\frac{1}{(\varphi (b)-\varphi (a))}\int\limits_{\varphi (a)}^{\varphi
(b)}f\left( x\right) dx\right\vert   \label{c2} \\
&&  \notag \\
&\leq &\frac{(\varphi (b)-\varphi (a))}{2}\left( s+\frac{1}{2^{s+1}}\right)
\left( \frac{\left\vert f^{\prime }(\varphi (b))\right\vert +\left\vert
f^{\prime }(\varphi (a))\right\vert }{(s+1)(s+2)}\right) .  \notag
\end{eqnarray}
\end{corollary}

\begin{corollary}
Under the assumptions of Theorem \ref{t1} with $h(t)=1,\ t\in \left(
0,1\right) ,$ we have%
\begin{equation}
\left\vert \frac{f\left( \varphi (a)\right) +f\left( \varphi (b)\right) }{2}-%
\frac{1}{(\varphi (b)-\varphi (a))}\int\limits_{\varphi (a)}^{\varphi
(b)}f\left( x\right) dx\right\vert \leq \frac{(\varphi (b)-\varphi (a))}{2}%
\left( \frac{\left\vert f^{\prime }(\varphi (b))\right\vert +\left\vert
f^{\prime }(\varphi (a))\right\vert }{2}\right) .  \label{c3}
\end{equation}
\end{corollary}

\begin{remark}
(a) In the case $c\rightarrow 0$ and $\varphi (x)=x$ for all $x\in \left[ a,b%
\right] ,$ then inequality (\ref{c1}) coincide with he right sides of
Hermite-Hadamard inequality proved by Dragomir and Agarwal in (\cite%
{dragomir2}).

(b) In the case $c\rightarrow 0$ and $\varphi (x)=x$ for all $x\in \left[ a,b%
\right] ,$ then inequality (\ref{c2}) gives the right sides of
Hermite-Hadamard inequality proved by Dragomir and Agarwal in (\cite%
{dragomir2}).
\end{remark}

\begin{theorem}
\label{t2} Let $h:\left( 0,1\right) \rightarrow \left( 0,\infty \right) $ be
a given function. Let $f:I\subset 
%TCIMACRO{\U{211d} }%
%BeginExpansion
\mathbb{R}
%EndExpansion
\rightarrow 
%TCIMACRO{\U{211d} }%
%BeginExpansion
\mathbb{R}
%EndExpansion
$ be a differentiable mapping on $I^{0}$ where $a,b\in I$ with $a<b$ and
Lebesgue integrable function$.$ If $\left\vert f^{\prime }\right\vert ^{q}$
is strongly $\varphi _{h}-$convex with respect to $c>0$ for the continuous
function $\varphi :\left[ a,b\right] \rightarrow \left[ a,b\right] ,$ $%
\varphi (a)<\varphi (b),$ and%
\begin{equation*}
A=c^{q}(\varphi (b)-\varphi (a))^{2q}B(q+1,q+1)-\frac{c}{6}(\varphi
(b)-\varphi (a))^{2}>0,
\end{equation*}%
then the following inequality holds;%
\begin{eqnarray*}
&&\left\vert \frac{f\left( \varphi (a)\right) +f\left( \varphi (b)\right) }{2%
}-\frac{1}{(\varphi (b)-\varphi (a))}\int\limits_{\varphi (a)}^{\varphi
(b)}f\left( x\right) dx\right\vert  \\
&& \\
&\leq &\dfrac{(\varphi (b)-\varphi (a))}{2^{\frac{1}{q}}}\left( \frac{1}{p+1}%
\right) ^{\frac{1}{p}}\left( \left( \left\vert f^{\prime }(\varphi
(b))\right\vert ^{q}+\left\vert f^{\prime }(\varphi (a))\right\vert
^{q}\right) \left( \int\limits_{0}^{1}h(t)dt\right) +A\right) ^{\frac{1}{q}}
\end{eqnarray*}%
for all$\ t\in \left( 0,1\right) ,$ $q\geq 1$ where $B$ is a beta function.
\end{theorem}

\begin{proof}
From Lemma \ref{z1} and by using H\"{o}lder's integral inequality, we have%
\begin{eqnarray*}
&&\left\vert \frac{f\left( \varphi (a)\right) +f\left( \varphi (b)\right) }{2%
}-\frac{1}{(\varphi (b)-\varphi (a))}\int\limits_{\varphi (a)}^{\varphi
(b)}f\left( x\right) dx\right\vert  \\
&& \\
&\leq &\dfrac{(\varphi (b)-\varphi (a))}{2}\left(
\int\limits_{0}^{1}\left\vert 2t-1\right\vert ^{p}dt\right) ^{\frac{1}{p}} \\
&& \\
&&\times \left( \int\limits_{0}^{1}\left[ \left\vert f^{\prime }\left(
t\varphi (b)+\left( 1-t\right) \varphi (a)\right) \right\vert
+ct(1-t)(\varphi (b)-\varphi (a))^{2}\right] ^{q}dt\right) ^{\frac{1}{q}}.
\end{eqnarray*}%
Since $\left\vert f^{\prime }\right\vert ^{q}$ is strongly $\varphi _{h}-$%
convex on $\left[ a,b\right] \ $and using the following inequality 
\begin{equation*}
\left( u+v\right) ^{q}\leq 2^{q-1}(u^{q}+v^{q}),\ u,v>0,\ q>1,
\end{equation*}%
we get%
\begin{eqnarray*}
&&\left\vert \frac{f\left( \varphi (a)\right) +f\left( \varphi (b)\right) }{2%
}-\frac{1}{(\varphi (b)-\varphi (a))}\int\limits_{\varphi (a)}^{\varphi
(b)}f\left( x\right) dx\right\vert  \\
&& \\
&\leq &\dfrac{(\varphi (b)-\varphi (a))}{2^{\frac{1}{q}}}\left(
\int\limits_{0}^{1}\left\vert 2t-1\right\vert ^{p}dt\right) ^{\frac{1}{p}} \\
&& \\
&&\times \left( \int\limits_{0}^{1}\left[ \left\vert f^{\prime }\left(
t\varphi (b)+\left( 1-t\right) \varphi (a)\right) \right\vert ^{q}+\left(
ct(1-t)(\varphi (b)-\varphi (a))^{2}\right) ^{q}\right] dt\right) ^{\frac{1}{%
q}} \\
&& \\
&\leq &\dfrac{(\varphi (b)-\varphi (a))}{2^{\frac{1}{q}}}\left(
\int\limits_{0}^{1}\left\vert 2t-1\right\vert ^{p}dt\right) ^{\frac{1}{p}} \\
&& \\
&&\times \left( \int\limits_{0}^{1}\left( h(t)\left\vert f^{\prime }(\varphi
(b))\right\vert ^{q}+h(1-t)\left\vert f^{\prime }(\varphi (a))\right\vert
^{q}\right. \right.  \\
&& \\
&&\left. -\left. ct(1-t)(\varphi (b)-\varphi (a))^{2}+\left[ ct(1-t)(\varphi
(b)-\varphi (a))^{2}\right] ^{q}\right) dt\right) ^{\frac{1}{q}}.
\end{eqnarray*}%
Thus, with simple calculations we obtain%
\begin{equation*}
\int\limits_{0}^{1}\left\vert 2t-1\right\vert ^{p}dt=\frac{1}{p+1},
\end{equation*}%
\begin{equation*}
\int\limits_{0}^{1}t(1-t)dt=\frac{1}{6},\ \ \
\int\limits_{0}^{1}t^{q}(1-t)^{q}dt=B(q+1,q+1)
\end{equation*}%
and%
\begin{equation*}
\int\limits_{0}^{1}h(t)dt=\int\limits_{0}^{1}h(1-t)dt.
\end{equation*}%
Therefore, we obtain%
\begin{eqnarray*}
&&\left\vert \frac{f\left( \varphi (a)\right) +f\left( \varphi (b)\right) }{2%
}-\frac{1}{(\varphi (b)-\varphi (a))}\int\limits_{\varphi (a)}^{\varphi
(b)}f\left( x\right) dx\right\vert  \\
&& \\
&\leq &\dfrac{(\varphi (b)-\varphi (a))}{2^{\frac{1}{q}}}\left( \frac{1}{p+1}%
\right) ^{\frac{1}{p}}\left( \left( \left\vert f^{\prime }(\varphi
(b))\right\vert ^{q}+\left\vert f^{\prime }(\varphi (a))\right\vert
^{q}\right) \left( \int\limits_{0}^{1}h(t)dt\right) \right.  \\
&& \\
&&+\left. c^{q}(\varphi (b)-\varphi (a))^{2q}B(q+1,q+1)-\frac{c}{6}(\varphi
(b)-\varphi (a))^{2}\right) ^{\frac{1}{q}}
\end{eqnarray*}%
which completes the proof.
\end{proof}

The following inequalities are associated the right side of Hermite-Hadamard
type inequalities for strongly $\varphi $-convex, strongly $\varphi _{s}$%
-convex strongly $\varphi -P$-convex with respect to $c>0,$ respectively.

\begin{corollary}
Under the assumptions of Theorem \ref{t2} with $h(t)=t,\ t\in \left(
0,1\right) ,$ we have%
\begin{eqnarray}
&&\left\vert \frac{f\left( \varphi (a)\right) +f\left( \varphi (b)\right) }{2%
}-\frac{1}{(\varphi (b)-\varphi (a))}\int\limits_{\varphi (a)}^{\varphi
(b)}f\left( x\right) dx\right\vert   \label{k1} \\
&&  \notag \\
&\leq &\dfrac{(\varphi (b)-\varphi (a))}{2^{\frac{1}{q}}}\left( \frac{1}{p+1}%
\right) ^{\frac{1}{p}}\left( \frac{\left\vert f^{\prime }(\varphi
(b))\right\vert ^{q}+\left\vert f^{\prime }(\varphi (a))\right\vert ^{q}}{2}%
+A\right) ^{\frac{1}{q}}.  \notag
\end{eqnarray}
\end{corollary}

\begin{corollary}
Under the assumptions of Theorem \ref{t2} with $h(t)=t^{s}$ $(s\in \left(
0,1\right) ),\ t\in \left( 0,1\right) ,$ we have%
\begin{eqnarray}
&&\left\vert \frac{f\left( \varphi (a)\right) +f\left( \varphi (b)\right) }{2%
}-\frac{1}{(\varphi (b)-\varphi (a))}\int\limits_{\varphi (a)}^{\varphi
(b)}f\left( x\right) dx\right\vert   \label{k2} \\
&&  \notag \\
&\leq &\dfrac{(\varphi (b)-\varphi (a))}{2^{\frac{1}{q}}}\left( \frac{1}{p+1}%
\right) ^{\frac{1}{p}}\left( \frac{\left\vert f^{\prime }(\varphi
(b))\right\vert ^{q}+\left\vert f^{\prime }(\varphi (a))\right\vert ^{q}}{s+1%
}+A\right) ^{\frac{1}{q}}.  \notag
\end{eqnarray}
\end{corollary}

\begin{corollary}
Under the assumptions of Theorem \ref{t2} with $h(t)=1,\ t\in \left(
0,1\right) ,$ we have%
\begin{eqnarray}
&&\left\vert \frac{f\left( \varphi (a)\right) +f\left( \varphi (b)\right) }{2%
}-\frac{1}{(\varphi (b)-\varphi (a))}\int\limits_{\varphi (a)}^{\varphi
(b)}f\left( x\right) dx\right\vert   \label{k3} \\
&&  \notag \\
&\leq &\dfrac{(\varphi (b)-\varphi (a))}{2^{\frac{1}{q}}}\left( \frac{1}{p+1}%
\right) ^{\frac{1}{p}}\left( \left\vert f^{\prime }(\varphi (b))\right\vert
^{q}+\left\vert f^{\prime }(\varphi (a))\right\vert ^{q}+A\right) ^{\frac{1}{%
q}}.  \notag
\end{eqnarray}
\end{corollary}

\begin{remark}
In the case $c\rightarrow 0$, inequalities (\ref{k1}) (\ref{k2}) and (\ref%
{k3}) reduce the right sides of Hermite-Hadamard type inequality for $%
\varphi $-convex, $\varphi _{s}$-convex and $\varphi -P$-convex functions$,$
respectively.
\end{remark}

\begin{lemma}
\label{z2} Let $f:I\subset 
%TCIMACRO{\U{211d} }%
%BeginExpansion
\mathbb{R}
%EndExpansion
\rightarrow 
%TCIMACRO{\U{211d} }%
%BeginExpansion
\mathbb{R}
%EndExpansion
$ be a differentiable mapping on $I^{0}$ where $a,b\in I$ with $a<b$ and $%
\varphi :\left[ a,b\right] \rightarrow \left[ a,b\right] .$ If $f^{\prime }\ 
$is Lebesgue integrable function, then the following equality holds;%
\begin{eqnarray*}
&&\frac{1}{(\varphi (b)-\varphi (a))}\int\limits_{\varphi (a)}^{\varphi
(b)}f\left( x\right) dx-f\left( \frac{\varphi (a)+\varphi (b)}{2}\right)  \\
&& \\
&=&(\varphi (b)-\varphi (a))\left\{ \int\limits_{0}^{\frac{1}{2}}t\left[
f^{\prime }\left( t\varphi (a)+\left( 1-t\right) \varphi (b)\right)
+ct(1-t)(\varphi (b)-\varphi (a))^{2}\right] dt\right.  \\
&& \\
&&\left. \int\limits_{\frac{1}{2}}^{1}\left( t-1\right) \left[ f^{\prime
}\left( t\varphi (a)+\left( 1-t\right) \varphi (b)\right) +ct(1-t)(\varphi
(b)-\varphi (a))^{2}\right] dt\right\} .
\end{eqnarray*}
\end{lemma}

\begin{proof}
By integration by parts, we can state:%
\begin{eqnarray}
J_{1} &=&\int\limits_{0}^{\frac{1}{2}}t\left[ f^{\prime }\left( t\varphi
(a)+\left( 1-t\right) \varphi (b)\right) +ct(1-t)(\varphi (b)-\varphi
(a))^{2}\right] dt  \notag \\
&&  \notag \\
&=&t\frac{f\left( t\varphi (a)+\left( 1-t\right) \varphi (b)\right) }{%
(\varphi (a)-\varphi (b))}\underset{0}{\overset{\frac{1}{2}}{\mid }}  \notag
\\
&&  \notag \\
&&-\frac{1}{(\varphi (a)-\varphi (b))}\int\limits_{0}^{\frac{1}{2}}f\left(
t\varphi (a)+\left( 1-t\right) \varphi (b)\right) dt+c(\varphi (b)-\varphi
(a))^{2}\int\limits_{0}^{\frac{1}{2}}t^{2}(1-t)dt  \label{3} \\
&&  \notag \\
&=&\frac{1}{2(\varphi (a)-\varphi (b))}f\left( \frac{\varphi (b)+\varphi (a)%
}{2}\right)   \notag \\
&&  \notag \\
&&+\frac{1}{(\varphi (b)-\varphi (a))}\int\limits_{0}^{\frac{1}{2}}f\left(
t\varphi (a)+\left( 1-t\right) \varphi (b)\right) dt+\frac{5c}{3\times 2^{6}}%
(\varphi (b)-\varphi (a))^{2}.  \notag
\end{eqnarray}%
and similarly%
\begin{eqnarray}
J_{2} &=&\int\limits_{\frac{1}{2}}^{1}\left( t-1\right) \left[ f^{\prime
}\left( t\varphi (a)+\left( 1-t\right) \varphi (b)\right) +ct(1-t)(\varphi
(b)-\varphi (a))^{2}\right] dt  \notag \\
&&  \notag \\
&=&\left( t-1\right) \frac{f\left( t\varphi (a)+\left( 1-t\right) \varphi
(b)\right) }{(\varphi (a)-\varphi (b))}\underset{\frac{1}{2}}{\overset{1}{%
\mid }}  \notag \\
&&  \notag \\
&&-\frac{1}{(\varphi (a)-\varphi (b))}\int\limits_{\frac{1}{2}}^{1}f\left(
t\varphi (a)+\left( 1-t\right) \varphi (b)\right) dt-c(\varphi (b)-\varphi
(a))^{2}\int\limits_{\frac{1}{2}}^{1}t(1-t)^{2}dt  \label{4} \\
&&  \notag \\
&=&\frac{1}{2(\varphi (a)-\varphi (b))}f\left( \frac{\varphi (b)+\varphi (a)%
}{2}\right)   \notag \\
&&  \notag \\
&&+\frac{1}{(\varphi (b)-\varphi (a))}\int\limits_{\frac{1}{2}}^{1}f\left(
t\varphi (a)+\left( 1-t\right) \varphi (b)\right) dt-\frac{5c}{3\times 2^{6}}%
(\varphi (b)-\varphi (a))^{2}.  \notag
\end{eqnarray}%
Adding (\ref{3}) and (\ref{4}) and rewritting, we easily deduce%
\begin{equation}
J=J_{1}+J_{2}=\frac{1}{(\varphi (b)-\varphi (a))}\int\limits_{0}^{1}f\left(
t\varphi (a)+\left( 1-t\right) \varphi (b)\right) dt-\frac{1}{(\varphi
(b)-\varphi (a))}f\left( \frac{\varphi (b)+\varphi (a)}{2}\right) .
\label{5}
\end{equation}%
Using the change of the variable $x=t\varphi (a)+\left( 1-t\right) \varphi
(b)$ for $t\in \left[ 0,1\right] ,$ and multiplying the both sides of (\ref%
{5}) by $(\varphi (b)-\varphi (a))$, we obtain%
\begin{equation*}
(\varphi (b)-\varphi (a))J=\frac{1}{(\varphi (b)-\varphi (a))}%
\int\limits_{\varphi (a)}^{\varphi (b)}f\left( x\right) dx-f\left( \frac{%
\varphi (a)+\varphi (b)}{2}\right) 
\end{equation*}%
which is required.
\end{proof}

\begin{theorem}
\label{t3} Let $h:\left( 0,1\right) \rightarrow \left( 0,\infty \right) $ be
a given function. Let $f:I\subset 
%TCIMACRO{\U{211d} }%
%BeginExpansion
\mathbb{R}
%EndExpansion
\rightarrow 
%TCIMACRO{\U{211d} }%
%BeginExpansion
\mathbb{R}
%EndExpansion
$ be a differentiable mapping on $I^{0}$ where $a,b\in I$ with $a<b$ and
Lebesgue integrable function$.$ If $\left\vert f^{\prime }\right\vert $ is
strongly $\varphi _{h}-$convex with respect to $c>0$ for the continuous
function $\varphi :\left[ a,b\right] \rightarrow \left[ a,b\right] $ and $%
\varphi (a)<\varphi (b),$ then the following inequality holds;%
\begin{eqnarray*}
&&\left\vert \frac{1}{(\varphi (b)-\varphi (a))}\int\limits_{\varphi
(a)}^{\varphi (b)}f\left( x\right) dx-f\left( \frac{\varphi (a)+\varphi (b)}{%
2}\right) \right\vert  \\
&& \\
&\leq &(\varphi (b)-\varphi (a))\left( \left\vert f^{\prime }(\varphi
(a))\right\vert +\left\vert f^{\prime }(\varphi (b))\right\vert \right)
\left( \int\limits_{0}^{\frac{1}{2}}t\left[ h(t)+h(1-t)\right] dt\right) 
\end{eqnarray*}%
for all $t\in \left( 0,1\right) .$
\end{theorem}

\begin{proof}
From Lemma \ref{z2} and by using strongly $\varphi _{h}-$convexity functions
with modulus $c>0$ of $\left\vert f^{\prime }\right\vert ,$ we have%
\begin{eqnarray*}
&&\left\vert \frac{1}{(\varphi (b)-\varphi (a))}\int\limits_{\varphi
(a)}^{\varphi (b)}f\left( x\right) dx-f\left( \frac{\varphi (a)+\varphi (b)}{%
2}\right) \right\vert  \\
&& \\
&\leq &(\varphi (b)-\varphi (a))\left\{ \int\limits_{0}^{\frac{1}{2}}t\left[
h(t)\left\vert f^{\prime }(\varphi (a))\right\vert +h(1-t)\left\vert
f^{\prime }(\varphi (b))\right\vert \right] dt\right.  \\
&& \\
&&+\left. \int\limits_{\frac{1}{2}}^{1}\left( 1-t\right) \left[
h(t)\left\vert f^{\prime }(\varphi (a))\right\vert +h(1-t)\left\vert
f^{\prime }(\varphi (b))\right\vert \right] dt\right\}  \\
&& \\
&=&(\varphi (b)-\varphi (a))\left( \left\vert f^{\prime }(\varphi
(a))\right\vert +\left\vert f^{\prime }(\varphi (b))\right\vert \right)
\left( \int\limits_{0}^{\frac{1}{2}}t\left[ h(t)+h(1-t)\right] dt\right) 
\end{eqnarray*}%
where using the fact that%
\begin{equation*}
\int\limits_{0}^{\frac{1}{2}}th(t)dt=\int\limits_{\frac{1}{2}}^{1}\left(
1-t\right) h(1-t)dt
\end{equation*}%
and%
\begin{equation*}
\int\limits_{0}^{\frac{1}{2}}th(1-t)dt=\int\limits_{\frac{1}{2}}^{1}\left(
1-t\right) h(t)dt
\end{equation*}%
which completes the proof.
\end{proof}

The following inequalities are associated the left side of Hermite-Hadamard
type inequalities for strongly $\varphi $-convex, strongly $\varphi _{s}$%
-convex strongly $\varphi -P$-convex with respect to $c>0,$ respectively.

\begin{corollary}
Under the assumptions of Theorem \ref{t3} with $h(t)=t,\ t\in \left(
0,1\right) ,$ we have%
\begin{eqnarray}
&&\left\vert \frac{1}{(\varphi (b)-\varphi (a))}\int\limits_{\varphi
(a)}^{\varphi (b)}f\left( x\right) dx-f\left( \frac{\varphi (a)+\varphi (b)}{%
2}\right) \right\vert   \notag \\
&&  \label{c4} \\
&\leq &(\varphi (b)-\varphi (a))\left( \frac{\left\vert f^{\prime }(\varphi
(a))\right\vert +\left\vert f^{\prime }(\varphi (b))\right\vert }{8}\right) .
\notag
\end{eqnarray}
\end{corollary}

\begin{corollary}
Under the assumptions of Theorem \ref{t3} with $h(t)=t^{s}$ $(s\in \left(
0,1\right) ),\ t\in \left( 0,1\right) ,$ we have%
\begin{eqnarray}
&&\left\vert \frac{1}{(\varphi (b)-\varphi (a))}\int\limits_{\varphi
(a)}^{\varphi (b)}f\left( x\right) dx-f\left( \frac{\varphi (a)+\varphi (b)}{%
2}\right) \right\vert   \notag \\
&&  \label{c5} \\
&\leq &(\varphi (b)-\varphi (a))\left( 1+\frac{s+3}{2^{s+2}}\right) \left( 
\frac{\left\vert f^{\prime }(\varphi (a))\right\vert +\left\vert f^{\prime
}(\varphi (b))\right\vert }{(s+1)(s+2)}\right) .  \notag
\end{eqnarray}
\end{corollary}

\begin{corollary}
Under the assumptions of Theorem \ref{t3} with $h(t)=1,\ t\in \left(
0,1\right) ,$ we have%
\begin{eqnarray}
&&\left\vert \frac{1}{(\varphi (b)-\varphi (a))}\int\limits_{\varphi
(a)}^{\varphi (b)}f\left( x\right) dx-f\left( \frac{\varphi (a)+\varphi (b)}{%
2}\right) \right\vert   \label{c6} \\
&&  \notag \\
&\leq &(\varphi (b)-\varphi (a))\left( \frac{\left\vert f^{\prime }(\varphi
(a))\right\vert +\left\vert f^{\prime }(\varphi (b))\right\vert }{4}\right) .
\notag
\end{eqnarray}
\end{corollary}

\begin{remark}
(a) In the case $c\rightarrow 0$ and $\varphi (x)=x$ for all $x\in \left[ a,b%
\right] ,$ then inequality (\ref{c4}) gives the right sides of
Hermite-Hadamard inequality proved by Kirmaci in (\cite{USK}).

(b) In the case $c\rightarrow 0$ and $\varphi (x)=x$ for all $x\in \left[ a,b%
\right] ,$ then inequality (\ref{c5}) coincide with he right sides of
Hermite-Hadamard inequality proved by Dragomir and Agarwal in (\cite%
{dragomir2}).
\end{remark}

\begin{theorem}
\label{t4} Let $h:\left( 0,1\right) \rightarrow \left( 0,\infty \right) $ be
a given function. Let $f:I\subset 
%TCIMACRO{\U{211d} }%
%BeginExpansion
\mathbb{R}
%EndExpansion
\rightarrow 
%TCIMACRO{\U{211d} }%
%BeginExpansion
\mathbb{R}
%EndExpansion
$ be a differentiable mapping on $I^{0}$ where $a,b\in I$ with $a<b$ and
Lebesgue integrable function$.$ If $\left\vert f^{\prime }\right\vert ^{q}$
is strongly $\varphi _{h}-$convex with respect to $c>0$ for the continuous
function $\varphi :\left[ a,b\right] \rightarrow \left[ a,b\right] $, $%
\varphi (a)<\varphi (b),$ and%
\begin{equation*}
G=c(\varphi (b)-\varphi (a))^{2q}B_{\frac{1}{2}}(q+1,q+1)-\frac{c}{12}%
(\varphi (b)-\varphi (a))^{2}>0,
\end{equation*}%
then the following inequality holds;%
\begin{eqnarray*}
&&\left\vert \frac{1}{(\varphi (b)-\varphi (a))}\int\limits_{\varphi
(a)}^{\varphi (b)}f\left( x\right) dx-f\left( \frac{\varphi (a)+\varphi (b)}{%
2}\right) \right\vert  \\
&& \\
&\leq &\dfrac{(\varphi (b)-\varphi (a))}{2^{\frac{1}{q}}}\left( \frac{1}{p+1}%
\right) ^{\frac{1}{p}}\left\{ \left( \int\limits_{0}^{\frac{1}{2}}\left(
h(t)\left\vert f^{\prime }(\varphi (a))\right\vert ^{q}+h(1-t)\left\vert
f^{\prime }(\varphi (b))\right\vert ^{q}\right) dt+G\right) ^{\frac{1}{q}%
}\right.  \\
&& \\
&&+\left. \left( \int\limits_{\frac{1}{2}}^{1}\left( h(t)\left\vert
f^{\prime }(\varphi (a))\right\vert ^{q}+h(1-t)\left\vert f^{\prime
}(\varphi (b))\right\vert ^{q}\right) dt+G\right) ^{\frac{1}{q}}\right\} 
\end{eqnarray*}%
for all $t\in \left( 0,1\right) ,$ $q>1,\ \frac{1}{p}+\frac{1}{q}=1$ where $%
B_{r}(.,.)$ is incomplete beta function.
\end{theorem}

\begin{proof}
From Lemma \ref{z2} and by using the H\"{o}lder's inequality, we have%
\begin{eqnarray*}
&&\left\vert \frac{1}{(\varphi (b)-\varphi (a))}\int\limits_{\varphi
(a)}^{\varphi (b)}f\left( x\right) dx-f\left( \frac{\varphi (a)+\varphi (b)}{%
2}\right) \right\vert  \\
&& \\
&\leq &(\varphi (b)-\varphi (a)) \\
&& \\
&&\times \left\{ \left( \int\limits_{0}^{\frac{1}{2}}t^{p}dt\right) ^{\frac{1%
}{p}}\left( \int\limits_{0}^{\frac{1}{2}}\left[ f^{\prime }\left( t\varphi
(a)+\left( 1-t\right) \varphi (b)\right) +ct(1-t)(\varphi (b)-\varphi
(a))^{2}\right] ^{q}dt\right) ^{\frac{1}{q}}\right.  \\
&& \\
&&+\left. \left( \int\limits_{\frac{1}{2}}^{1}\left( t-1\right)
^{p}dt\right) ^{\frac{1}{p}}\left( \int\limits_{\frac{1}{2}}^{1}\left[
f^{\prime }\left( t\varphi (a)+\left( 1-t\right) \varphi (b)\right)
+ct(1-t)(\varphi (b)-\varphi (a))^{2}\right] dt\right) ^{\frac{1}{q}%
}\right\} .
\end{eqnarray*}%
Since $\left\vert f^{\prime }\right\vert ^{q}$ is strongly $\varphi _{h}-$%
convex on $\left[ a,b\right] \ $and using the following inequality 
\begin{equation*}
\left( u+v\right) ^{q}\leq 2^{q-1}(u^{q}+v^{q}),\ u,v>0,\ q>1,
\end{equation*}%
we get%
\begin{eqnarray*}
&&\left\vert \frac{1}{(\varphi (b)-\varphi (a))}\int\limits_{\varphi
(a)}^{\varphi (b)}f\left( x\right) dx-f\left( \frac{\varphi (b)+\varphi (a)}{%
2}\right) \right\vert  \\
&\leq &2^{1-\frac{1}{q}}(\varphi (b)-\varphi (a))\left\{ \left(
\int\limits_{0}^{\frac{1}{2}}t^{p}dt\right) ^{\frac{1}{p}}\left(
\int\limits_{0}^{\frac{1}{2}}\left[ \left\vert f^{\prime }\left( t\varphi
(a)+\left( 1-t\right) \varphi (b)\right) \right\vert ^{q}+\left(
ct(1-t)(\varphi (b)-\varphi (a))^{2}\right) ^{q}\right] dt\right) ^{\frac{1}{%
q}}\right.  \\
&&+\left. \left( \int\limits_{\frac{1}{2}}^{1}\left( 1-t\right)
^{p}dt\right) ^{\frac{1}{p}}\left( \int\limits_{\frac{1}{2}}^{1}\left[
\left\vert f^{\prime }\left( t\varphi (a)+\left( 1-t\right) \varphi
(b)\right) \right\vert ^{q}+\left( ct(1-t)(\varphi (b)-\varphi
(a))^{2}\right) ^{q}\right] dt\right) ^{\frac{1}{q}}\right\}  \\
&\leq &\frac{(\varphi (b)-\varphi (a))}{2}\left( \frac{1}{p+1}\right) ^{%
\frac{1}{p}} \\
&&\times \left\{ \left( \int\limits_{0}^{\frac{1}{2}}\left[ h(t)\left\vert
f^{\prime }(\varphi (a))\right\vert ^{q}+h(1-t)\left\vert f^{\prime
}(\varphi (b))\right\vert ^{q}\right. \right. \right.  \\
&&\left. \left. -ct(1-t)(\varphi (b)-\varphi (a))^{2}+\left( ct(1-t)(\varphi
(b)-\varphi (a))^{2}\right) ^{q}\right] dt\right) ^{\frac{1}{q}} \\
&&+\left( \int\limits_{\frac{1}{2}}^{1}\left[ h(t)\left\vert f^{\prime
}(\varphi (a))\right\vert ^{q}+h(1-t)\left\vert f^{\prime }(\varphi
(b))\right\vert ^{q}\right. \right.  \\
&&-\left. \left. \left. ct(1-t)(\varphi (b)-\varphi (a))^{2}+\left(
ct(1-t)(\varphi (b)-\varphi (a))^{2}\right) ^{q}dt\right] \right) ^{\frac{1}{%
q}}\right\}  \\
&=&\frac{(\varphi (b)-\varphi (a))}{2}\left( \frac{1}{p+1}\right) ^{\frac{1}{%
p}} \\
&&\times \left\{ \left( \int\limits_{0}^{\frac{1}{2}}\left( h(t)\left\vert
f^{\prime }(\varphi (a))\right\vert ^{q}+h(1-t)\left\vert f^{\prime
}(\varphi (b))\right\vert ^{q}\right) dt\right. \right.  \\
&&\left. -\frac{c}{12}(\varphi (b)-\varphi (a))^{2}+c(\varphi (b)-\varphi
(a))^{2q}B_{\frac{1}{2}}(q+1,q+1)\right) ^{\frac{1}{q}} \\
&&+\left( \int\limits_{\frac{1}{2}}^{1}\left( h(t)\left\vert f^{\prime
}(\varphi (a))\right\vert ^{q}+h(1-t)\left\vert f^{\prime }(\varphi
(b))\right\vert ^{q}\right) dt\right.  \\
&&-\left. \left. \frac{c}{12}(\varphi (b)-\varphi (a))^{2}+c(\varphi
(b)-\varphi (a))^{2q}B_{\frac{1}{2}}(q+1,q+1)\right) ^{\frac{1}{q}}\right\} .
\end{eqnarray*}%
Thus, with simple calculations we obtain%
\begin{equation*}
\int\limits_{0}^{\frac{1}{2}}t^{p}dt=\int\limits_{\frac{1}{2}}^{1}\left(
1-t\right) ^{p}dt=\frac{1}{2^{p+1}\left( p+1\right) },
\end{equation*}%
\begin{equation*}
\int\limits_{0}^{\frac{1}{2}}t(1-t)dt=\int\limits_{\frac{1}{2}}^{1}t(1-t)dt=%
\frac{1}{12},
\end{equation*}%
\begin{equation*}
\int\limits_{0}^{\frac{1}{2}}t^{q}(1-t)^{q}dt=\int\limits_{\frac{1}{2}%
}^{1}t^{q}(1-t)^{q}dt=B_{\frac{1}{2}}(q+1,q+1).
\end{equation*}%
Therefore, using the above obtained results, we have%
\begin{eqnarray*}
&&\left\vert \frac{1}{(\varphi (b)-\varphi (a))}\int\limits_{\varphi
(a)}^{\varphi (b)}f\left( x\right) dx-f\left( \frac{\varphi (a)+\varphi (b)}{%
2}\right) \right\vert  \\
&\leq &\frac{(\varphi (b)-\varphi (a))}{2}\left( \frac{1}{p+1}\right) ^{%
\frac{1}{p}}\left( \left[ \int\limits_{0}^{\frac{1}{2}}\left( h(t)\left\vert
f^{\prime }(\varphi (a))\right\vert ^{q}+h(1-t)\left\vert f^{\prime
}(\varphi (b))\right\vert ^{q}\right) dt+G\right] ^{\frac{1}{q}}\right.  \\
&&+\left. \left[ \int\limits_{\frac{1}{2}}^{1}\left( h(t)\left\vert
f^{\prime }(\varphi (a))\right\vert ^{q}+h(1-t)\left\vert f^{\prime
}(\varphi (b))\right\vert ^{q}\right) dt+G\right] ^{\frac{1}{q}}\right) 
\end{eqnarray*}%
which completes the proof.
\end{proof}

The following inequalities are associated the left side of Hermite-Hadamard
type inequalities for strongly $\varphi $-convex, strongly $\varphi _{s}$%
-convex strongly $\varphi -P$-convex with respect to $c>0,$ respectively.

\begin{corollary}
Under the assumptions of Theorem \ref{t4} with $h(t)=t,\ t\in \left(
0,1\right) ,$ we have%
\begin{eqnarray}
&&\left\vert \frac{1}{(\varphi (b)-\varphi (a))}\int\limits_{\varphi
(a)}^{\varphi (b)}f\left( x\right) dx-f\left( \frac{\varphi (a)+\varphi (b)}{%
2}\right) \right\vert   \notag \\
&&  \notag \\
&\leq &\dfrac{(\varphi (b)-\varphi (a))}{2^{\frac{1}{q}}}\left( \frac{1}{p+1}%
\right) ^{\frac{1}{p}}  \label{10} \\
&&  \notag \\
&&\times \left\{ \left( \frac{\left\vert f^{\prime }(\varphi (a))\right\vert
^{q}+3\left\vert f^{\prime }(\varphi (b))\right\vert ^{q}}{8}+G\right) ^{%
\frac{1}{q}}+\left( \frac{3\left\vert f^{\prime }(\varphi (a))\right\vert
^{q}+\left\vert f^{\prime }(\varphi (b))\right\vert ^{q}}{8}+G\right) ^{%
\frac{1}{q}}\right\} .  \notag
\end{eqnarray}
\end{corollary}

\begin{corollary}
Under the assumptions of Theorem \ref{t4} with $h(t)=t^{s}$ $(s\in \left(
0,1\right) ),\ t\in \left( 0,1\right) ,$ we have%
\begin{eqnarray}
&&\left\vert \frac{1}{(\varphi (b)-\varphi (a))}\int\limits_{\varphi
(a)}^{\varphi (b)}f\left( x\right) dx-f\left( \frac{\varphi (a)+\varphi (b)}{%
2}\right) \right\vert   \notag \\
&&  \notag \\
&\leq &\dfrac{(\varphi (b)-\varphi (a))}{2^{\frac{1}{q}}}\left( \frac{1}{p+1}%
\right) ^{\frac{1}{p}}  \label{20} \\
&&  \notag \\
&&\times \left\{ \left( \frac{1}{2^{s+1}\left( s+1\right) }\left\vert
f^{\prime }(\varphi (a))\right\vert ^{q}+\frac{1}{\left( s+1\right) }\left(
1-\frac{1}{2^{s+1}}\right) \left\vert f^{\prime }(\varphi (b))\right\vert
^{q}+G\right) ^{\frac{1}{q}}\right.   \notag \\
&&  \notag \\
&&+\left. \left( \frac{1}{\left( s+1\right) }\left( 1-\frac{1}{2^{s+1}}%
\right) \left\vert f^{\prime }(\varphi (a))\right\vert ^{q}+\frac{1}{%
2^{s+1}\left( s+1\right) }\left\vert f^{\prime }(\varphi (b))\right\vert
^{q}+G\right) ^{\frac{1}{q}}\right\} .  \notag
\end{eqnarray}
\end{corollary}

\begin{corollary}
Under the assumptions of Theorem \ref{t4} with $h(t)=1,\ t\in \left(
0,1\right) ,$ we have%
\begin{eqnarray}
&&\left\vert \frac{1}{(\varphi (b)-\varphi (a))}\int\limits_{\varphi
(a)}^{\varphi (b)}f\left( x\right) dx-f\left( \frac{\varphi (a)+\varphi (b)}{%
2}\right) \right\vert   \notag \\
&&  \label{30} \\
&\leq &\dfrac{(\varphi (b)-\varphi (a))}{2^{\frac{1}{q}-1}}\left( \frac{1}{%
p+1}\right) ^{\frac{1}{p}}\left( \frac{\left\vert f^{\prime }(\varphi
(b))\right\vert ^{q}+\left\vert f^{\prime }(\varphi (a))\right\vert ^{q}}{2}%
+G\right) ^{\frac{1}{q}}.  \notag
\end{eqnarray}
\end{corollary}

\begin{remark}
In the case $c\rightarrow 0$, inequalities (\ref{10}) (\ref{20}) and (\ref%
{30}) reduce the right sides of Hermite-Hadamard type inequality for $%
\varphi $-convex, $\varphi _{s}$-convex and $\varphi -P$-convex functions$,$
respectively.
\end{remark}

\end{document}